\pdfoutput=1
\documentclass{article}
\usepackage[english]{babel}
\usepackage[a4paper,scale=.68]{geometry}
\usepackage{amsmath,amssymb}
\usepackage{amsthm,thmtools}
\usepackage{enumitem}
\usepackage{xcolor}
\usepackage{lmodern}
\usepackage{microtype}
\usepackage{hyperref}

\author{Josse van Dobben de Bruyn}
\title{Almost all positive continuous linear functionals can~be~extended}
\date{28 April 2021} 

\makeatletter
\newcommand{\myaddr}[1]{\gdef\my@address{\par\textsc{#1}}}
\newcommand{\mycuraddr}[1]{\gdef\my@curaddr{\par\textsc{#1}}}
\newcommand{\myemail}[1]{\gdef\my@email{\par\textit{E-mail address:} \texttt{\href{mailto:#1}{#1}.}}}
\AtEndDocument{\par\bigskip\my@address\my@curaddr\my@email}
\newcommand{\mythanks}[1]{\gdef\my@thanks{#1}}
\newcommand{\mysubjclass}[2][2010]{\gdef\my@subjclass{#1 \textit{Mathematics Subject Classification}. #2.}}
\newcommand{\mykeywords}[1]{\gdef\my@keywords{\textit{Key words and phrases}. #1.}}
\newcommand{\mymaketitle}{%
	\let\@oldauthor\@author%
	\gdef\@author{\@oldauthor\texorpdfstring{\footnotemark}{}}%
	\let\@oldthanks\@thanks%
	\gdef\@thanks{\@oldthanks\footnotetext[0]{\my@subjclass}\footnotetext[0]{\my@keywords}\footnotetext[\the\c@footnote]{\my@thanks}}%
	\maketitle
}
\makeatother

\mythanks{Partially supported by the Dutch Research Council (NWO), project number 613.009.127.}
\myaddr{Mathematical Institute, Leiden University, P.O.~Box~9512, 2300~RA~Leiden, The~Netherlands}
\mycuraddr{Delft Institute of Applied Mathematics, Delft University of Technology, Mekelweg~4, 2628~CD~Delft, The~Netherlands}
\myemail{J.vanDobbendeBruyn@tudelft.nl}
\mysubjclass[2020]{Primary: 46A40. Secondary: 47B65, 47L07, 52A20}
\mykeywords{Convex cone, partially ordered topological vector space, continuous positive linear functional, positive extension}

\def\mylinkcolor{black!75!blue}
\makeatletter
\hypersetup{pdfauthor=\@author,pdftitle=\@title,colorlinks=true,linkcolor=\mylinkcolor,citecolor=\mylinkcolor,urlcolor=\mylinkcolor}
\makeatother

\newcommand{\hair}{\ifmmode\mskip1mu\else\kern0.08em\fi}

\declaretheorem[style=definition,sibling=equation]{definition}
\declaretheorem[style=definition,sibling=definition]{situation}

\declaretheorem[style=remark,sibling=definition]{remark}

\declaretheorem[style=plain,sibling=definition]{theorem}
\declaretheorem[style=plain,sibling=definition]{corollary}
\declaretheorem[style=plain,sibling=definition]{proposition}

\declaretheorem[style=definition,numbered=no]{acknowledgement}

\newcommand{\R}{\ensuremath{\mathbb{R}}}

\newcommand{\algdual}{^*}
\newcommand{\topdual}{'}

\newcommand{\tensor}{\mathbin{\otimes}}

\newcommand{\mywedge}{\mathcal K}
\newcommand{\rwedge}{\mathcal R}
\newcommand{\boundary}{\partial}
\newcommand{\mybasis}{\mathcal B}
\DeclareMathOperator{\id}{id}

\begin{document}
\mymaketitle
\begin{abstract}
	Let $F$ be an ordered topological vector space (over $\R$) whose positive cone $F_+$ is weakly closed, and let $E \subseteq F$ be a subspace. We prove that the set of positive continuous linear functionals on $E$ that can be extended (positively and continuously) to $F$ is weak\nobreakdash-$*$ dense in the topological dual wedge $E_+\topdual$.
	Furthermore, we show that this result cannot be generalized to arbitrary positive operators, even in finite-dimensional spaces.
\end{abstract}

\bigskip
\noindent
This short note is about extensions of positive operators on a (pre-)ordered topological vector space over the real numbers. For notation and terminology, see \cite{Cones&Duality}.

Extension theorems for positive operators have been studied in great detail, and can be found in many textbooks on ordered vector spaces (e.g.~\cite[\S V.5]{Schaefer}, \cite[\S 1.4--1.5]{Cones&Duality}).
Many classical extension theorems provide sufficient conditions for a given operator to have a positive extension.
These include a classical extension theorem of Kantorovich (see \cite[Thm.~1.30]{Cones&Duality}), and a theorem of Lotz on positive operators from a Banach sublattice of a Banach lattice to an AL-space (see \cite[Prop.~3.2]{Lotz}, \cite[Thm.~II.8.9]{Schaefer-Banach-lattices}).

In general, not every positive operator can be extended, and the problem already arises when looking at functionals.
In 1957, Mirkil proved that, for a finite-dimensional ordered vector space $F$ whose positive cone $F_+$ is closed, all positive linear functionals on all subspaces of $F$ can be extended if and only if $F_+$ is polyhedral (see \cite[Cor.~1]{Mirkil}, \cite[Thm.~4.13]{Klee}).
That same year, Bauer and Namioka independently found a necessary and sufficient condition for a positive continuous linear functional $E \to \R$ defined on a subspace $E$ of a topological vector space $F$ to have a positive and continuous extension to $F$ (see \cite[Thm.~1]{Bauer}, \cite[Thm.~4.4]{Namioka}, \cite[Thm.~V.5.4]{Schaefer}).

In this note, we look at things from a different perspective.
Instead of determining whether or not a positive operator can be extended, we try to \emph{approximate} it by extendable positive operators.
We prove that the extendable positive linear functionals are weak-$*$ dense in the dual wedge.
Furthermore, we show that this result cannot be extended to arbitrary positive operators, even if the spaces are finite-dimensional.

Although the proofs in this note are relatively simple, the results appear to be unknown in the ordered vector spaces community.
The author has not been able to locate an earlier proof (or statement) of \autoref{thm:main} or \autoref{cor:almost-all} in the literature.

For positive functionals, we have the following positive result.

\begin{theorem}
	\label{thm:main}
	Let $F$ be a preordered topological vector space, and let $E \subseteq F$ be a subspace.
	If the topological dual $F\topdual$ separates points on $F$, and if the positive wedge $F_+$ is weakly closed, then the set of all positive continuous linear functionals on $E$ that can be extended positively and continuously to $F$ is weak\nobreakdash-$*$ dense in $E_+\topdual$.
\end{theorem}

We make two remarks about the statement of \autoref{thm:main}. First, if $F$ is locally convex, then we may replace the requirement that $F_+$ is weakly closed by the requirement that $F_+$ is closed. These two requirements are now equivalent, because $F_+$ is convex.

Second, if $F_+$ is weakly closed and if $F_+ \cap -F_+ = \{0\}$, then $\{0\}$ is also weakly closed, so the weak topology is Hausdorff. Hence, if $F_+$ is a cone, then the requirement that $F\topdual$ separates points on $F$ is automatically met. Thus, the statement from the abstract is recovered.


\begin{proof}[{Proof of \autoref{thm:main}}]
	Define $\rwedge := \{\varphi|_E \, : \, \varphi \in F_+\topdual\} \subseteq E_+\topdual$, and note that $\rwedge$ is the wedge of positive continuous linear functionals $E \to \R$ that can be extended positively and continuously to $F$.
	Since $F_+$ is weakly closed, we have
	\[ F_+ = \big\{x \in F \, : \, \langle x , \varphi \rangle \geq 0 \ \text{for all $\varphi \in F_+\topdual$}\big\} . \]
	It follows that
	\[ E_+ := F_+ \cap E = \big\{x \in E \, : \, \langle x , \varphi \rangle \geq 0 \ \text{for all $\varphi \in F_+\topdual$}\big\} . \]
	This shows that $E_+$ is the predual wedge of $\rwedge$. Hence, by the bipolar theorem, $E_+\topdual$ is the weak\nobreakdash-$*$ closure of $\rwedge$.
\end{proof}

In the finite-dimensional case, we have the following immediate corollary.

\begin{corollary}
	\label{cor:almost-all}
	Let $F$ be a finite-dimensional preordered vector space, and let $E \subseteq F$ be a subspace. If the positive wedge $F_+$ is closed, then $\lambda$-almost all positive linear functionals $E \to \R$ can be extended positively to $F$, where $\lambda$ denotes the Lebesgue \textup(i.e.~Haar\textup) measure on $E\algdual$.
\end{corollary}
\begin{proof}
	The set $\rwedge \subseteq E_+\algdual$ of extendable positive linear functionals is a convex set, and it follows from \autoref{thm:main} that $\overline{\rwedge} = E_+\algdual$.
	The result follows since $E_+\algdual \setminus \rwedge = \overline{\rwedge} \setminus \rwedge \subseteq \boundary \rwedge$, and the boundary of a convex set in $\R^n$ has Lebesgue measure zero (e.g.~\cite{Lang}).
\end{proof}



One might ask if similar results hold for arbitrary positive operators. Unfortunately, this is not the case, even if the spaces are finite-dimensional, as we will now demonstrate.

A wedge $\mywedge \subseteq E$ is a \emph{simplex cone} (or \emph{Yudin cone}) if there is an algebraic basis $\mybasis$ of $E$ such that $\mywedge$ is the wedge generated by $\mybasis$.

If $E$ and $G$ are vector spaces, and if $\varphi \in E\algdual$ and $z \in G$, then we write $z \tensor \varphi$ for the linear map $E \to G$, $x \mapsto \langle x, \varphi \rangle z$.

For our counterexample, we will consider the following situation.

\begin{situation}
	\label{sit:polyhedral-cone}
	Let $E$ be a finite-dimensional vector space, and let $E_+ \subseteq E$ be a generating polyhedral cone which is \emph{not} a simplex cone. Let $\varphi_1,\ldots,\varphi_m \in E_+\algdual$ be representatives of the extremal rays of $E_+\algdual$. Then $E_+ = \bigcap_{i=1}^m \{x \in E \, : \, \langle x , \varphi_i \rangle \geq 0\}$, and every positive linear functional is a positive combination of $\varphi_1,\ldots,\varphi_m$. Additionally, let $F := \R^m$ with the standard cone $F_+ := \R_{\geq 0}^m$, so that the map $T : E \to F$, $x \mapsto (\varphi_1(x),\ldots,\varphi_m(x))$ is bipositive. We will identify $E$ with a subspace of $F$ via this map.
\end{situation}

\begin{proposition}
	\label{prop:extendable-vs-projective-cone}
	In \autoref{sit:polyhedral-cone}, the positive linear maps $E \to E$ that can be extended to a positive linear map $F \to E$ are precisely the maps of the form $\sum_{i=1}^k x_i \tensor \psi_i$ with $x_1,\ldots,x_k \in E_+$ and $\psi_1,\ldots,\psi_k \in E_+\algdual$.
\end{proposition}
\begin{proof}
	Let $e_1,\ldots,e_m$ denote the standard basis of $F = \R^m$.
	If $T : E \to E$ is a positive linear map that can be extended to a positive linear map $S : F \to E$, then we have
	\[ T(x) = S(\varphi_1(x),\ldots,\varphi_m(x)) = S(e_1) \varphi_1(x) + \cdots + S(e_m) \varphi_m(x), \]
	so $T$ can be written as $T = \sum_{j=1}^m S(e_j) \tensor \varphi_j$.
	
	Conversely, suppose that $T = \sum_{i=1}^k x_i \tensor \psi_i$ with $x_1,\ldots,x_k \in E_+$ and $\psi_1,\ldots,\psi_k \in E_+\algdual$. Every $\psi_i$ can be written as a positive combination of the $\varphi_1,\ldots,\varphi_m$, so after rearranging the terms we may write $T = \sum_{j=1}^m y_j \tensor \varphi_j$, where the $y_j$ are positive combinations of the $x_i$. In particular, $y_1,\ldots,y_m \in E_+$. Therefore the map $S : F \to E$, $e_j \mapsto y_j$ is a positive extension of $T$.
\end{proof}

The following theorem of Barker and Loewy tells us that approximation by operators of the form described in \autoref{prop:extendable-vs-projective-cone} is not always possible.

\begin{theorem}[{Barker--Loewy, \cite[Proposition 3.1]{Barker-Loewy}}]
	\label{thm:Barker-Loewy}
	Let $E$ be a finite-dimensional ordered vector space whose positive cone $E_+$ is closed and generating. Then the identity $\id_E : E \to E$ can be written as the limit of a sequence of operators of the form $\sum_{i=1}^k x_i \tensor \psi_i$ \textup(with $x_1,\ldots,x_k \in E_+$ and $\psi_1,\ldots,\psi_k \in E_+\algdual$\textup) if and only if $E_+$ is a simplex cone.
\end{theorem}

In particular, if $E$ and $F$ are as in \autoref{sit:polyhedral-cone}, then it follows from \autoref{prop:extendable-vs-projective-cone} and \autoref{thm:Barker-Loewy} that the identity $\id_E : E \to E$ cannot be approximated by positive operators that can be extended to positive operators $F \to E$.

\begin{remark}
	Although \autoref{thm:Barker-Loewy} is exactly what Barker and Loewy proved and exactly what we used, it was later shown by Tam \cite[Thm.~1]{Tam} that the cone of operators of the form $\sum_{i=1}^k x_i \tensor \psi_i$ (with $x_1,\ldots,x_k \in E_+$ and $\psi_1,\ldots,\psi_k \in E_+\algdual$) is already closed.
	Therefore another equivalent condition in \autoref{thm:Barker-Loewy} is that $\id_E$ itself can be written in this positive tensor form.
	
	In \autoref{sit:polyhedral-cone}, an easier way to see that the cone generated by $\{x \tensor \psi \, : \, x \in E_+,\, \psi \in E_+\algdual \}$ is closed is by noting that it is finitely generated.
	We conclude that the positive operators $E \to E$ which can be extended to a positive operator $F \to E$ form a closed cone which is strictly contained in the (closed) cone of all positive operators $E \to E$.
\end{remark}

\begin{acknowledgement}
	The author wishes to thank Marcel de Jeu for helpful comments and suggestions, which helped improve the clarity and brevity of this note.
\end{acknowledgement}

\phantomsection

\end{document}